\documentclass{amsart}
\pdfoutput=1

\usepackage{lmodern}
\usepackage[utf8]{inputenc}
\usepackage[mathscr]{eucal}
\usepackage{amssymb}
\usepackage{stmaryrd}
\usepackage[usenames,dvipsnames]{xcolor}
\usepackage{xspace}
\usepackage{amsthm}
\usepackage{bbold}
\usepackage[normalem]{ulem}
\usepackage{float}
\usepackage{adjustbox}
\usepackage{enumitem}
\setlist{leftmargin=*, wide, labelindent=0pt}
\setlist[enumerate]{label*=(\alph*),ref=\alph*}

\usepackage{array}
\usepackage{amsmath}
\usepackage{wasysym}
\usepackage{etoolbox}
\usepackage{mathdots}
\usepackage[style=alphabetic,maxbibnames=99,backend=bibtex8]{biblatex}

\numberwithin{equation}{section}
\usepackage[all]{xy}
\SelectTips{cm}{}
\newdir{ >}{{}*!/-10pt/\dir{>}}

\usepackage{xr-hyper}
\usepackage[colorlinks=true,linkcolor={Black},citecolor={Black},urlcolor={Black}]{hyperref}
\usepackage[nameinlink]{cleveref}

\crefname{Thm}{Theorem}{Theorems}
\crefname{Rem}{Remark}{Remarks}
\crefname{Prop}{Proposition}{Propositions}
\crefname{Cor}{Corollary}{Corollaries}
\crefname{Cons}{Construction}{Constructions}
\crefname{Exa}{Example}{Examples}
\crefname{Lem}{Lemma}{Lemmas}
\crefname{Rec}{Recollection}{Recollections}

\def\makeautorefname#1#2{\expandafter\def\csname#1autorefname\endcsname{#2}}
\def\equationautorefname~#1\null{(#1)\null}
\makeautorefname{footnote}{footnote}%
\makeautorefname{item}{item}%
\makeautorefname{figure}{Figure}%
\makeautorefname{table}{Table}%
\makeautorefname{part}{Part}%
\makeautorefname{appendix}{Appendix}%
\makeautorefname{chapter}{Chapter}%
\makeautorefname{section}{Section}%
\makeautorefname{subsection}{Section}%
\makeautorefname{subsubsection}{Section}%
\makeautorefname{theorem}{Theorem}%
\makeautorefname{corollary}{Corollary}%
\makeautorefname{lemma}{Lemma}%
\makeautorefname{proposition}{Proposition}%
\makeautorefname{pro}{Property}
\makeautorefname{conj}{Conjecture}%
\makeautorefname{definition}{Definition}%
\makeautorefname{notn}{Notation}
\makeautorefname{notns}{Notations}
\makeautorefname{rem}{Remark}%
\makeautorefname{quest}{Question}%
\makeautorefname{example}{Example}%
\theoremstyle{plain}  
\newtheorem{theorem}{Theorem}[section]
\newtheorem{proposition}{Proposition}[section]
\newtheorem{lemma}{Lemma}[section]
\newtheorem{corollary}{Corollary}[section]

\theoremstyle{definition}
\newtheorem{definition}{Definition}[section]
\newtheorem{notation}{Notation}[section]
\newtheorem{example}{Example}[section]
\newtheorem{remark}{Remark}[section]

\makeatletter
\let\c@corollary=\c@theorem
\let\c@proposition=\c@theorem
\let\c@lemma=\c@theorem
\let\c@definition=\c@theorem
\let\c@example=\c@theorem
\let\c@remark=\c@theorem
\numberwithin{equation}{section}

\newcommand{\nc}{\newcommand}
\nc{\dmo}{\DeclareMathOperator}

\dmo{\coker}{coker}
\dmo{\cone}{cone}
\dmo{\Der}{D}
\dmo{\DAM}{\DAMbig^{\geom}}%
\dmo{\DAMbig}{DAM}%
\dmo{\Ext}{Ext}
\dmo{\Gal}{Gal}
\dmo{\Hm}{H}
\dmo{\Hom}{Hom}
\dmo{\Id}{Id}
\dmo{\Ind}{Ind}
\dmo{\Infl}{Infl}
\dmo{\Ker}{Ker}
\dmo{\Mod}{Mod}
\dmo{\opname}{op}
\dmo{\perm}{perm}
\dmo{\Perm}{Perm}
\dmo{\SH}{SH}
\dmo{\SHmot}{SH^{\mathrm{c}}_{\bbA^{\!1}}}
\dmo{\Spc}{Spc}
\dmo{\Spec}{Spec}
\dmo{\Spech}{\Spec^{h}}
\dmo{\stab}{stab}
\dmo{\Stab}{Stab}
\dmo{\supp}{supp}
\dmo{\thick}{thick}
\dmo{\Locname}{Loc}

\nc{\Inj}{\mathrm{Inj}}
\nc{\cat}[1]{\mathscr{#1}}
\nc{\cK}{\cat{K}}
\nc{\cL}{\cat{L}}
\nc{\colim}{\mathop{\mathrm{colim}}}
\nc{\cofib}{\mathop{\mathrm{cofib}}}
\nc{\cP}{\cat{P}}
\nc{\cQ}{\cat{Q}}
\nc{\cS}{\cat{S}}
\nc{\cT}{\cat{T}}
\nc{\CAlg}{\mathrm{CAlg}}
\nc{\eg}{{\sl e.g.}\@\xspace}
\nc{\gp}{\mathfrak{p}}
\nc{\gq}{\mathfrak{q}}
\nc{\hook}{\hookrightarrow}
\nc{\ie}{{\sl i.e.}\@\xspace}
\nc{\into}{\mathop{\rightarrowtail}}
\nc{\inv}{^{-1}}
\nc{\kk}{k}
\nc{\kkG}{\kk G}
\nc{\Loc}[1]{\Locname(#1)}
\nc{\loccit}{{\sl loc.\ cit.}\xspace}
\nc{\Mid}{\,\big|\,}
\nc{\onto}{\mathop{\twoheadrightarrow}}
\nc{\op}{^{\opname}}
\nc{\sminus}{\smallsetminus}
\nc{\potimes}[1]{^{\otimes #1}}
\nc{\sbull}{{\scriptscriptstyle\bullet}}
\nc{\SET}[2]{\big\{\,#1\Mid#2\,\big\}}
\nc{\unit}{\mathbb{1}}
\newcommand{\bb}[1]{\mathbb{#1}}

\newcommand{\mc}[1]{\mathcal{#1}}
\newcommand{\mf}[1]{\mathfrak{#1}}
\newcommand{\mo}[1]{\operatorname{#1}}
\nc{\W}{\mathbb{W}}

\dmo{\sta}{sta}
\nc{\rsd}[1]{\mathrm{rsd}_{#1}}

\dmo{\End}{End}
\nc{\isoto}{\overset{\sim}{\,\to\,}}
\nc{\isofrom}{\overset{\sim}{\,\leftarrow\,}}
\let\le=\leqslant

\nc{\sto}{\rightsquigarrow}
\nc{\xisoto}[1]{\xrightarrow[\sim]{#1}}
\nc{\xto}[1]{\xrightarrow{#1}}
\nc{\xfrom}[1]{\xleftarrow{#1}}
\nc{\xinto}[1]{\overset{#1}{\,\into\,}}
\nc{\xonto}[1]{\overset{#1}{\,\onto\,}}
\nc{\lto}{\leftarrow}
\nc{\normaleq}{\trianglelefteqslant}

\nc{\normal}{\lhd}
\nc{\normaleop}{\mathop{\mathring{\trianglelefteqslant}}}
\dmo{\chara}{char}%
\dmo{\CoInd}{CoInd}
\dmo{\DMbig}{DM}
\dmo{\id}{id}
\dmo{\Img}{Im}
\dmo{\im}{im}
\dmo{\Komp}{K}
\dmo{\proj}{proj}
\dmo{\rmH}{H}
\dmo{\Res}{Res}
\dmo{\smallb}{b}
\dmo{\geom}{gm}
\dmo{\stabname}{stab}
\dmo{\comp}{comp}
\dmo{\Supp}{Supp}
\dmo{\kosname}{kos}
\dmo{\subname}{Sub}

\nc{\Sub}[1]{\subname_{#1}}
\nc{\Weyl}[2]{{#1}/\!\!/{#2}}
\nc{\WGH}{\Weyl{G}{H}}
\nc{\tInd}{{}^{\otimes\!}\Ind}
\nc{\inn}{;}
\nc{\Vee}[1]{V_{#1}}%
\nc{\Loctens}[1]{\Locname_{\otimes}(#1)}
\nc{\SpcKG}{\Spc(\cK(G))}
\nc{\SpcKGk}{\Spc(\cK(G;\kk))}
\nc{\SpcKE}{\Spc(\cK(E))}
\nc{\Rall}{\rmH^{\sbull\sbull}}
\nc{\EA}[2]{\mathcal{E}_{#1}(#2)}
\nc{\EApp}[1]{\EA{p}{#1}}

\nc{\kos}[2][]{\kosname_{#1}(#2)}
\nc{\sKG}{\kos[G]{K}}

\nc{\Zp}{\hat{\bbZ}_p}

\nc{\adh}[1]{\overline{#1}}
\nc{\adj}{\dashv}
\nc{\apriori}{{\sl a priori}\xspace}
\nc{\bs}{\backslash}
\nc{\cf}{{\sl cf.}\ }
\nc{\Db}{\Der_{\smallb}}
\nc{\D}{\Der}
\nc{\FFsep}{\overline{\FF}}
\nc{\Fp}{\bbF_{\!p}}
\nc{\gm}{\mathfrak{m}}
\mathchardef\mhyphen="2D
\nc{\ideal}[1]{\langle #1\rangle}
\nc{\Kb}{\Komp_{\smallb}}
\nc{\K}{\Komp}
\nc{\leop}{\mathop{\mathring{\le}}}
\nc{\Lotimes}{\otimes^{\rmL}}
\nc{\To}{\Rightarrow}
\nc{\cSpec}{\mathsf{Spec}}
\nc{\Top}{\mathsf{Top}}

\let\theoldbibliography\thebibliography
\renewcommand{\thebibliography}[1]{%
	\theoldbibliography{#1}%
	\setlength{\parskip}{0ex}
	\setlength{\itemsep}{0.5ex plus 0.2ex minus 0.2ex}
	\small
}
\apptocmd{\thebibliography}{\raggedright}{}{}
\font\maljapanese=dmjhira at 2.5ex
\newcommand{\yo}{\textrm{\!\maljapanese\char"48}}
\usepackage{tikz}

\usetikzlibrary{patterns}
\usepackage{tikz-cd}

\usepackage{mathtools}
\usepackage{xspace}
\makeatletter
\let\ea\expandafter
\newcount\foreachcount

\def\foreachLetter#1#2#3{\foreachcount=#1
  \ea\loop\ea\ea\ea#3\@Alph\foreachcount
  \advance\foreachcount by 1
  \ifnum\foreachcount<#2\repeat}

\def\definebb#1{\ea\gdef\csname #1#1\endcsname{\ensuremath{\mathbb{#1}}\xspace}}
\foreachLetter{1}{27}{\definebb}

\makeatother

\date{October 8, 2024}
\author{Logan Hyslop}
\email{loganrhyslop@gmail.com}
\urladdr{https://loganhyslop.github.io/}

\hypersetup{pdfauthor={Logan Hyslop},pdftitle={Towards the Nerves of Steel Conjecture}}

\bibliography{Bibliography.bib}
\begin{document}
	\DeclareDocumentCommand\rart{ g }{%
	{\ar[r, tail]%
		\IfNoValueF {#1} { \ar[r, tail, "#1"]}%
	}%
}
\DeclareDocumentCommand\dart{ g }{%
	{\ar[d, tail]%
		\IfNoValueF {#1} { \ar[d, tail, "#1"]}%
	}%
}
\DeclareDocumentCommand\rarh{ g }{%
	{\ar[r, two heads]%
		\IfNoValueF {#1} { \ar[r, two heads, "#1"]}%
	}%
}
\DeclareDocumentCommand\darh{ g }{%
	{\ar[d, two heads]%
		\IfNoValueF {#1} { \ar[d, two heads, "#1"]}%
	}%
}


\title{Towards the Nerves of Steel Conjecture}

\begin{abstract}
Given a local $\otimes$-triangulated category, and  a fiber sequence $y\xrightarrow{g} \unit \xrightarrow{f} x$, one may ask if there is always a nonzero object $z$ such that either $z\otimes f$ or $z\otimes g$ is $\otimes$-nilpotent.  The claim that this property holds for all local $\otimes$-triangulated categories is equivalent to Balmer's ``nerves of steel conjecture'' \cite[~Remark 5.15]{Balmer2020}.  In the present paper, we will see how this property can fail if the category we start with is not rigid, discuss a large class of categories where the property holds, and ultimately prove that the nerves of steel conjecture is equivalent to a stronger form of this property.
\end{abstract}

\subjclass[2024]{18F99; 55P43}
\keywords{Tensor triangular geometry, Nerves of Steel, spectrum, rational $\bb{E}_{\infty}$-ring.}

\maketitle

\section{Introduction}
\label{sec:intro}

The nilpotence theorem of Devinatz-Hopkins-Smith \cite{DevHopSmith} paved the way towards chromatic homotopy theory, which was a major example motivating Balmer's theory of tensor triangular geometry introduced in \cite{Balmer2005}.  Naively stated, the nerves of steel conjecture aims to strengthen the connection between tensor triangular geometry and ``nilpotence theorems.''  In order to make this precise, we must first make a few definitions.
\begin{definition}(\cite[~Remark 3.4]{Balmer2020})
To any (essentially small) symmetric monoidal stable $\infty$-category $\mc{C}$, the homological spectrum $\mo{Spc}^h(\mc{C})$ is defined with points being maximal Serre $\otimes$-ideals in the (abelian) category of finitely presented objects of the Yoneda category~$\mo{Fun}^{add}(\mc{C}^{op},\mo{Ab})$.
\end{definition}
There is a natural comparison map $\mo{Spc}^h(\mc{C})\to\mo{Spc}(\mc{C})$ from the homological spectrum of $\mc{C}$ to its Balmer spectrum.  In \cite[~Corollary 3.9]{Balmer2020}, Balmer proves that when $\mc{C}$ is rigid, the comparison map is surjective, and that the homological spectrum is essentially determined by a property of detecting nilpotence in a certain precise sense (see \cite[~Corollary 4.7]{Balmer2020} and \cite[~Theorem 5.4]{Balmer2020}).

Before proceeding, let's introduce some definitions which will be used throughout the paper.
\begin{definition}\label{def1.2}
	When discussing a \textit{symmetric monoidal stable $\infty$-category} $\mc{C}$, we will implicitly assume that $\mc{C}$ is essentially small and idempotent complete.  That is, we identify the $\infty$-category of symmetric monoidal stable $\infty$-categories with the category $2\text{-}\mo{Ring}=\mo{CAlg}(\mo{Cat}^{\mo{st}}_{\infty})$ from \cite[~Definition 2.14]{mathew2016galoisgroupstablehomotopy}.
\end{definition}
\begin{definition}
	We say that a symmetric monoidal stable $\infty$-category $\mc{C}$ is \textit{local} if it is nonzero, and given $x,y\in\mc{C}$ with $x\otimes y\simeq 0$, then either $x$ or $y$ is $\otimes$-nilpotent (if $\mc{C}$ is rigid, this is equivalent to saying that either $x\simeq 0$ or $y\simeq 0$).
\end{definition}
Only in \textsection 2 will we deal with idempotent-complete symmetric monoidal stable $\infty$-categories which are possibly not rigid.  Thus, instead of repeating ``rigid symmetric monoidal stable $\infty$-category'' all the time, we will simply use the term $\otimes$-triangular category, or tt-category, to indicate rigidity.  
\begin{definition}
	The term \textit{tt-category} will refer to an idempotent-complete rigid symmetric monoidal stable $\infty$-category, identifying the $\infty$-category of tt-categories with the full subcategory $$\mo{CAlg}(\mo{Cat}^{\mo{st}}_{\infty})^{\mo{rig}}\subseteq \mo{CAlg}(\mo{Cat}^{\mo{st}}_{\infty}).$$ A \textit{local tt-category} will mean a tt-category which is also local.
\end{definition}

In \cite[~Remark 5.15]{Balmer2020}, Balmer had the nerves of steel to avoid conjecturing that the comparison map between the homological spectrum and the Balmer spectrum of a tt-category $\mc{T}$ is always an isomorphism.  Despite his best efforts, this statement has come to be known as Balmer's ``nerves of steel conjecture.''  The problem we study takes a slightly different form, but is in much the same spirit.  To state it, we require one last definition.
\begin{definition}\label{def1.1}
We say that the \textit{exact-nilpotence condition} holds for a local tt-category $\mc{T}$ if whenever we have a fiber sequence $$y\xrightarrow{g}\unit\xrightarrow{f}x,$$ there exists a nonzero object $z\in \mc{T}$ such that either $z\otimes g$ or $z\otimes f$ is $\otimes$-nilpotent.
\end{definition}

The connection between this property and the nerves of steel conjecture follows from further work of Balmer summarized in the following theorem.
\begin{theorem}[\cite{balmer2020homological}]\label{th1.2}
The nerves of steel conjecture holds if and only if the exact-nilpotence condition holds for every local tt-category~$\mc{T}$.
\end{theorem}
\begin{proof}
It follows from \cite[~Theorem A.1]{balmer2020homological} that the nerves of steel conjecture is equivalent to the statement that in any local tt-category $\mc{T}$, given any morphisms $f,g$ with $f\otimes g\simeq0$, there is some nonzero object $z$ with $z\otimes f$ or $z\otimes g$ is $\otimes$-nilpotent.  Assuming the nerves of steel conjecture, given any fiber sequence as in Definition~\ref{def1.1} in a local tt-category $\mc{T}$, one has that $g\otimes f\simeq 0$, thus we find there exists some nonzero object $z\in\mc{T}$ with $z\otimes f$ or $z\otimes g$ $\otimes$-nilpotent, thus the exact-nilpotence condition holds.  Conversely, assume the exact-nilpotence condition always holds, and take any $g,f\in \mc{T}$ with~$f\otimes g\simeq 0$.  We may assume without loss of generality that $g:\unit\to y$ and~$f:\unit\to x$.  Then, $(f\otimes id_{y})\circ g\simeq 0$, so that $g$ factors over $\mo{fib}(f)\otimes y\to y$, and the exact-nilpotence condition applied to the fiber sequence $\mo{fib}(f)\to \unit\to x$ proves the claim.
\end{proof}

In \textsection 2, we will show that if one drops the requirement that $\mc{T}$ be rigid, the exact-nilpotence condition can fail.  The nerves of steel conjecture was formulated only in the rigid case, and the connection between the homological spectrum and the exact-nilpotence condition is also a feature of the rigid case.  Nevertheless, these examples provide clues to how the exact-nilpotence condition could potentially fail.  We proceed by looking at free constructions.  In particular, we prove,
\begin{theorem}
If $\mc{T}$ is the subcategory of compact objects in the free stably symmetric monoidal stable $\infty$-category on an object with a map from the unit over any local tt-category, then $\mc{T}$ is a local category for which the exact-nilpotence condition is false.
\end{theorem}

On the positive side, we prove in \textsection 3 that the exact-nilpotence condition holds for a large class of local tt-categories which are generated by their unit.  The class of such tt-categories is closed under filtered colimits with local transition maps.  We will discuss many categories where the condition is known to hold, and add onto this by showing, for instance:
\begin{proposition}
	If $R$ is a connective $\bb{E}_{\infty}$-ring such that $\pi_0(R)$ is a local ring, then $\mc{C}=\mo{Mod}_{R}^{perf}$ is a local tt-category.  If $\pi_*(R)$ is a Noetherian ring, then the exact-nilpotence condition holds for~$\mc{C}$.
\end{proposition}
This extends to showing that the exact-nilpotence condition holds for the category $\mo{Mod}_R^{perf}$ over any connective rational $\bb{E}_{\infty}$-ring $R$ with $\pi_0(R)$ a local ring, the proof of which takes the majority of the section.

In the final section, we slightly strengthen Theorem~\ref{th1.2}, to the following claim
\begin{theorem}
The following are equivalent,
\begin{enumerate}
	\item The nerves of steel conjecture holds.
	\item For every local tt-category $\mc{T}$, the exact-nilpotence condition holds.
	\item There exists an integer $n$ such that for every local tt-category $\mc{T}$, and any fiber sequence as in Definition~\ref{def1.1}, there exists a nonzero object $z\in\mc{T}$ such that either $z\otimes g^{\otimes n}\simeq 0$ or~$z\otimes f^{\otimes n}\simeq 0$.
\end{enumerate}
\end{theorem}

\subsection{Conventions}
For the purposes of this paper, we will often use $\infty$-categorical language, following Lurie \cite{lurie2008highertopostheory}\cite{HA}.  We list off some definitions that will be used throughout the paper.

\begin{definition}
A symmetric monoidal stable functor $F:\mc{C}\to\mc{D}$ between local tt-categories will itself be called \textit{local} if given $c\in \mc{C}$, $F(c)\simeq 0$ iff~$c\simeq 0$.
\end{definition}
\begin{remark}
A local functor between local tt-categories is nothing more than a conservative functor as defined above.  However, if one were studying the functor on the full big categories $\mo{Ind}(F):\mo{Ind}(\mc{C})\to \mo{Ind}(\mc{D})$, this functor is far from conservative, yet still deserves to be termed local as it is conservative when restricted to compact-rigid objects.
\end{remark}

When working with $\bb{E}_{\infty}$-rings over a characteristic zero field $k$, we will denote by $\Lambda_k[var_{2n+1}]$ the free $\bb{E}_{\infty}$-$k$-algebra on a class in degree $2n+1$, where $var$ will be a variable name.  We similarly denote by $k[var_{2n}]$ the free $\bb{E}_{\infty}$-$k$-algebra on a class in degree~$2n$.  If $R$ is an arbitrary $\bb{E}_{\infty}$-$k$-algebra with a class $y_{2n}\in \pi_{2n}(R)$, we denote by $R/y_{2n}$ the ring $R\otimes_{k[y_{2n}]}k$, which identifies with the cofiber~$\mo{cofib}(y_{2n}:\Sigma^{2n}R\to R)$.
\begin{definition}
When working with an $\bb{E}_{\infty}$-ring $R$, we write $\mo{Mod}_R^{perf}$ for the category of compact objects in the category $\mo{Mod}_R$ of $R$-module spectra.  We often will abbreviate and write $\mo{Spc}(R)$ for the Balmer spectrum $\mo{Spc}(\mo{Mod}_R^{perf})$ of this category.
\end{definition}

\subsection{Acknowledgments} I would like to thank Paul Balmer for recommending this problem to me and for guiding me as I worked on this project.  I would also like to thank Tobias Barthel for reading over the draft and providing useful comments; and Akhil Mathew for answering some of my questions related to finitely presented rational $\bb{E}_{\infty}$-rings.

\newpage
\section{The Non-Rigid Case}
We begin this section with a review of universal constructions to motivate what is to come.  Throughout this section, we work in the $\infty$-category $2\text{-}\mo{Ring}$ from Definition~\ref{def1.2}.  In the following, when working in the category $\mo{Fin}_*$ of finite pointed sets, we write $\langle n\rangle$ for the pointed set $\{*\}\coprod\{1,\ldots,n\}$, pointed at~$*$.
\subsection{Universal constructions.}  Suppose we are given an (essentially small) $\infty$-operad $\mc{O}^{\otimes}$, and a symmetric monoidal stable $\infty$-category $\mc{C}$.  Our goal is to construct a universal symmetric monoidal stable $\infty$-category $\mc{C}[\mc{O}]$, together with a symmetric monoidal exact functor from $\mc{C}$, such that for any $\mc{C}$-algebra $\mc{D}\in 2\text{-}\mo{Ring}_{\mc{C}/}$ (that is, symmetric monoidal stable $\infty$-category $\mc{D}$ with a distinguished symmetric monoidal exact functor from $\mc{C}$), there is an equivalence of $\infty$-categories $$\mo{Fun}_{2\text{-}\mo{Ring}_{\mc{C}/}}(\mc{C}[\mc{O}],\mc{D})\simeq \mo{Alg}_{\mc{O}}(\mc{D}),$$ between symmetric monoidal exact $\mc{C}$-linear functors from $\mc{C}[\mc{O}]$ to $\mc{D}$ and $\mc{O}$-algebra objects in $\mc{D}$.  We proceed as follows.

Given a symmetric monoidal $\infty$-category $\mc{C}$, pointed objects $\unit\to x$ in $\mc{C}$ may be viewed as algebra objects over the $\infty$-operad~$\mo{Poi}^{\otimes}\simeq \bb{E}_0$.  This operad is defined as the subcategory of finite pointed sets $\mo{Fin}_*$ containing all objects, and morphisms are maps $f:*\coprod S \to *\coprod T$ such that $|f^{-1}(t)|\leq 1$ for all $t\in T$ \cite[~Example 2.1.1.19]{HA}.  As in \cite[~Construction 2.2.4.1]{HA}, we can form the monoidal envelope $$\mo{Env}(\bb{E}_0):=\bb{E}_0\times_{\mo{Fun}(\{0\},\mo{Fin}_*)}\mo{Act}(\mo{Fin}_*)_{/\langle 1\rangle},$$ where $\mo{Act}(\mo{Fin}_*)_{/\langle 1\rangle}$ is the category of active morphisms in $\mo{Fin}_*$ with target~$\{*\}\coprod \{1\}$.\footnote{Recall that a morphism $f$ is active if $f^{-1}(*)=*$.}  There is a unique active morphism from any object to $\langle 1\rangle$, which provides us with an identification of $\mo{Env}(\bb{E}_0)$ with the category $\mo{Fin}^{inj}$ of finite sets and injective maps between them.  Summarizing, we have:
\begin{lemma}
The free symmetric monoidal $\infty$-category on a pointed object is given by the category $\mo{Fin}^{inj}$ of finite sets with injective maps.
\end{lemma}
\begin{proof}
By \cite[~Proposition 2.2.4.4]{HA}, $\mo{Env}(\mc{O})$ is a symmetric monoidal $\infty$-category, in our case, giving that $\mo{Fin}^{inj}$ is symmetric monoidal with monoidal structure given by taking coproducts.  Using \cite[~Proposition 2.2.4.9]{HA}, the $\infty$-category of symmetric monoidal functors $\mo{Env}(\mc{O})\to\mc{D}$ landing in a symmetric monoidal $\infty$-category $\mc{D}$ is equivalent to the category of $\mc{O}$-algebra objects in $\mc{D}$.  Now use that $\mo{Alg}_{\bb{E}_0}(\mc{D})$ is the category of pointed objects in $\mc{D}$ see e.g. \cite[~Proposition 2.1.3.9]{HA}.
\end{proof}

Now, to arrive at the free situation in tt geometry, we proceed as follows.  Let $\mc{T}$ be an idempotent complete symmetric monoidal stable $\infty$-category.  By \cite[~Lemma 5.3.2.11]{HA}, there is an equivalence between $2\text{-}\mo{Ring}$ and the category $\mo{CAlg}(\mo{Pr}^{L,\omega}_{\mo{st}})$ of commutative algebra objects in compactly generated presentable stable $\infty$-categories (equipped with the Lurie tensor product).  This equivalence leads us to study the following:
\begin{lemma}[\cite{HA}]
Given an essentially small $\infty$-operad $\mc{O}$ and some $\mc{C}\in 2\text{-}\mo{Ring}$, the free symmetric monoidal stable $\infty$-category over $\mc{C}$ with an $\mc{O}$-algebra is given by the presheaf category $\mo{Fun}(\mo{Env}(\mc{O})^{op},\mo{Ind}(\mc{C}))^{\omega}$, together with the Day convolution tensor product structure.
\end{lemma}
\begin{proof}
Using \cite[~Proposition 4.8.1.10]{HA}, we find that the free presentably symmetric monoidal $\infty$-category (presentable as it is a presheaf category) on $\mo{Env}(\mc{O})$ is given by the presheaf category~$\mc{P}(\mo{Env}(\mc{O}))\simeq \mo{Fun}(\mo{Env}(\mc{O})^{op},\mc{S})$.  Since the Yoneda image gives a collection of compact generators of this $\infty$-category, it in fact lives in $\mo{CAlg}(\mo{Pr}^{L,\omega})$, and one quickly sees that it is also the free compactly generated presentably symmetric monoidal $\infty$-category on $\mo{Env}(\mc{O})$ (free in the sense that the $\infty$-category of functors out of this category is equivalent to $\mc{O}$-algebras in the compact objects of the target).  Now, by \cite[~Proposition 3.2.4.7]{HA}, the coproduct in $\mo{CAlg}(\mo{Pr}^{L,\omega})$ is given by the Lurie tensor product, so that for a given $\mc{C}\in 2\text{-}\mo{Ring}$, the free $\mo{Ind}(\mc{C})$-algebra in $\mo{Pr}^{L,\omega}$ on $\mc{O}$ is given by $$ \mc{P}(\mo{Env}(\mc{O}))\otimes \mo{Ind}(\mc{C}).$$  Using \cite[~Proposition 4.8.1.17]{HA}, we find that
\begin{align*}
\mc{P}(\mo{Env}(\mc{O}))\otimes \mo{Ind}(\mc{C})&\simeq \mo{RFun}(\mc{P}(\mo{Env}(\mc{O}))^{op},\mo{Ind}(\mc{C}))\\
&\simeq \mo{LFun}(\mc{P}(\mo{Env}(\mc{O})),\mo{Ind}(\mc{C})^{op})^{op}\\
&\simeq \mo{Fun}(\mo{Env}(\mc{O}),\mo{Ind}(\mc{C})^{op})^{op}\\
&\simeq \mo{Fun}(\mo{Env}(\mc{O})^{op},\mo{Ind}(\mc{C})).
\end{align*}
Finally, we pass by the equivalence $2\text{-}\mo{Ring}\simeq \mo{CAlg}(\mo{Pr}^{L,\omega}_{\mo{st}})$ via passing to compact objects in order to conclude.
\end{proof}
Specializing to our case of interest,
\begin{corollary}
The free idempotent complete symmetric monoidal stable $\infty$-category over such a category $\mc{T}$ is the category $$ \mo{Fun}((\mo{Fin}^{inj})^{op},\mo{Ind}(\mc{T}))^{\omega}$$ of compact objects in~$\mo{Fun}((\mo{Fin}^{inj})^{op},\mo{Ind}(\mc{T}))$, equipped with the Day convolution tensor product.
\end{corollary}\qed\\
It will be convenient to introduce the following notation.
\begin{notation}	Given a small symmetric monoidal $\infty$-category $\mc{C}$, denote the category $\mo{Fun}((\mo{Fin}^{inj})^{op},\mo{Ind}(\mc{C}))^{\omega}$ considered above by~$\mc{C}[\mo{Poi}]$.  Additionally, we denote by $\mc{C}[X]$ the category $\mo{Fun}((\mo{Fin}^{\simeq})^{op},\mo{Ind}(\mc{C}))^{\omega}$ of compact objects in $\mo{Ind}(\mc{C})$-valued presheaves on the category of finite sets and bijective maps.
\end{notation}
\begin{remark}
The second category appearing above, $\mc{C}[X]$, can be shown to be the free symmetric monoidal stable $\infty$-category over $\mc{C}$ on an object.  This serves as a sort of ``non-rigid'' tt-affine line, which the notation was chosen to reflect.
\end{remark}

\subsection{The Counterexample}  
We come to the main theorem of this section. 
\begin{theorem}\label{th2.5}
Let $\mc{C}$ be a local symmetric monoidal stable $\infty$-category.  Then the category $\mc{C}[\mo{Poi}]$ described in the previous section is a local symmetric monoidal stable $\infty$-category for which the exact-nilpotence condition fails.
\end{theorem}
Before proving this theorem, we must first embark on a journey to study the ``affine line'' $\mc{C}[X]$ over~$\mc{C}$.  This has appeared in \cite[~Example 2.17]{mathew2016galoisgroupstablehomotopy}, during Mathew's own discussion on free constructions.
\begin{proposition}\label{prop2.6}
If $\mc{C}$ is a local symmetric monoidal stable $\infty$-category, then so too is~$\mc{C}[X]$.
\end{proposition}
\begin{proof}
By construction as (compact objects in) a presheaf category on $\mo{FinSet}^{\simeq}$ valued in $\mo{Ind}(\mc{C})$, equipped with the Day convolution symmetric monoidal structure, $\mc{C}[X]$ is a symmetric monoidal stable $\infty$-category.  We now focus on showing that it is local.  Suppose given a presheaf $P:(\mo{FinSet}^{\simeq})^{op}\to \mo{Ind}(\mc{C})$ which is compact as an object of $\mo{Fun}((\mo{FinSet}^{\simeq})^{op},\mo{Ind}(\mc{C}))$, we then claim that $P([n])\simeq 0$ for all but finitely many $n$, with $[n]$ the $n$-element set $\{1,\ldots, n\}$.  Indeed, for $i\in\bb{N}$, defining $P_i([n])=P([n])$ for $n\leq i$, and $P_i([n])=0$ otherwise, we can write $P$ as the filtered colimit $P\simeq \varinjlim_{i}P_i$, and compactness of $P$ forces the identity of $P$ to factor over a finite stage, so that $P_i\to P$ splits for some $i$ (at which point $P=P_i$ by construction).  That is to say, the category $\mc{C}[X]$ decomposes as a sum $$\mc{C}[X]\simeq \coprod_{n\geq 0}\mo{Fun}(B\Sigma_n,\mo{Ind}(\mc{C}))^{\omega},$$ such that every object $f\in\mc{C}[X]$ factors as a (finite) direct sum $f\simeq \oplus_{n\geq 0}f_n$, where $f_n:(\mo{Fin}^{\simeq})^{op}\to\mc{C}$ are presheaves with the property that $f_n(S)=0$ if~$|S|\neq n$, and $f_n\simeq 0$ for almost all~$n\in\bb{N}$.  We use that for a finite group $G$, $\mo{Fun}(BG,\mo{Ind}(\mc{C}))^{\omega}\subseteq \mo{Fun}(BG,\mo{Ind}(\mc{C}))$ factors over $\mo{Fun}(BG,\mc{C})$, as the forgetful functor $\mo{Fun}(BG,\mo{Ind}(\mc{C}))\to \mo{Ind}(\mc{C})$ is left adjoint to coinduction, which commutes with filtered colimits, enforcing that the forgetful functor preserves compact objects.  This allows us to work with \textit{some} presheaves valued in $\mc{C}$ as opposed to $\mo{Ind}(\mc{C})$.  To show that $\mc{C}[X]$ is local, it suffices to show that any two nonzero objects of the form $f=f_n:B\Sigma_n\to \mc{C}$ and $g=g_m:B\Sigma_m\to\mc{C}$ have nonzero tensor product.  Given two such nonzero objects $f$ and $g$, the object $$f\otimes g:B\Sigma_{n}\times B\Sigma_{m}\to \mc{C}\times \mc{C}\xrightarrow{\otimes}\mc{C}$$ is nonzero, since $\mc{C}$ is local.  The tensor product $f\otimes_{\mc{C}[X]}g$, given by Day convolution, is described through inducing this representation~$\mo{Ind}_{B\Sigma_n\times B\Sigma_m}^{B\Sigma_{m+n}}(f\otimes g)$.

Thus, our claim reduces to showing that the functor $$\mo{Ind}_{B\Sigma_n\times B\Sigma_m}^{B\Sigma_{m+n}}:\mc{C}^{B\Sigma_{n}\times B\Sigma_{m}}\to \mc{C}^{B\Sigma_{n+m}}$$ is faithful on objects.  This follows from the description $$\mo{Ind}_{B\Sigma_n\times B\Sigma_m}^{B\Sigma_{m+n}}(-)\simeq \unit[\Sigma_{n+m}]\otimes_{\unit[\Sigma_n\times\Sigma_m]}-,$$ and the fact that $\unit[\Sigma_{n+m}]$ is a direct sum of copies of~$\unit[\Sigma_n\times\Sigma_m]$.
\end{proof}
\begin{remark}
For general $\mc{C}$, not every $\Sigma_n$-equivariant object of $\mc{C}$ is compact as an object of $\mo{Fun}(B\Sigma_n,\mo{Ind}(\mc{C}))$, the former category being possibly larger.  However, the above proof still goes through if we work with the larger category $\coprod_{n\geq 0}\mc{C}^{B\Sigma_n}$, which could be used equally well in place of $\mc{C}[X]$ for what follows.
\end{remark}
\begin{proof}[Proof of Theorem \ref{th2.5}]
First, note that the inclusion $\mo{Fin}^{\simeq}\to \mo{Fin}^{inj}$ gives rise to a functor~$\mo{Res}:\mc{C}[\mo{Poi}]\to\mc{C}[X]$.  It is easy to see that this functor exists at the level of presentable categories.  In order to see that it actually takes compact objects to compact objects, select a choice of compact generators $\{c_i\}_{i\in I}$ for~$\mc{C}$.  Then, the collection $\{\yo([n])\otimes_{\mc{C}} c_i\}_{i\in I, n\in\bb{N}}$ for naturals $n$ (where $[n]$ is again the $n$-element set $\{1,...,n\}$, and $\yo:\mo{Fin}^{inj}\to\mc{C}[\mo{Poi}]$ is the Yoneda embedding) forms a collection of compact generators for the category~$\mc{C}[\mo{Poi}]$.  In particular, $\mc{C}[\mo{Poi}]$ identifies with the thick stable closure of this collection in the category $\mo{Fun}((\mo{Fin}^{inj})^{op},\mo{Ind}{\mc{C}})$, so to see that any object of $\mc{C}[\mo{Poi}]$ lands in $\mc{C}[X]$, it suffices to test on these test objects, where it can be seen directly.

A priori, $\mo{Res}$ need not be symmetric monoidal.  However, in our case, it actually is.  Since this functor arises as the right adjoint of a symmetric monoidal functor (which classifies the free object of $\mc{C}[\mo{Poi}]$), $\mo{Res}$ is lax symmetric monoidal, so our goal is to show that this lax structure can be promoted to a genuine symmetric monoidal structure.  Consider the full subcategory $\mc{T}$ of $\mc{C}[\mo{Poi}]\times \mc{C}[\mo{Poi}]$ consisting of all pairs $(f,g)$ for which the natural map $\mo{Res}(f)\otimes \mo{Res}(g)\to \mo{Res}(f\otimes g)$ is an equivalence.  The category $\mc{T}$ is a thick stable subcategory of $\mc{C}[\mo{Poi}]$, so we can use our test objects $\{\yo([n])\otimes_{\mc{C}}c_i\}$ from before.  Note that $\mo{Res}(\yo([0]))\simeq \yo([0])$, such that $\mo{Res}$ takes the unit to the unit.  Using that $\yo([1])$ generates the Yoneda image under tensor powers, we can reduce to checking that tensor powers of $\yo([1])$ map to tensor powers of~$\mo{Res}(\yo([1]))$.  The object $\mo{Res}(\yo[1])$ is given as~$\yo[0]\oplus\yo[1]\in\mc{C}[X]$.  Thus, $$\mo{Res}(1)^{\otimes m}\simeq \oplus_{n\geq 0}\yo([n])^{\oplus {m\choose n}},$$ whereas, $$\mo{Res}((\yo([1]))^{\otimes m})\simeq \mo{Res}(\yo([m]))\simeq \oplus_{n\geq 0}\mo{Fin}^{inj}([n],[m])\simeq \oplus_{n\geq 0}\yo([n])^{\oplus {m\choose n}},$$ and the map provided by the lax symmetric monoidal structure is an equivalence.

Now, since $\mo{Res}$ is faithful on objects, we find that $\mc{C}[\mo{Poi}]$ is local by Proposition~\ref{prop2.6}.  Note that $\mo{Res}(\yo([0])\to \yo([1]))$ is split, so the free map cannot become $\otimes$-nilpotent after tensoring with a nonzero object of~$\mc{C}[\mo{Poi}]$.  Now, note that there is a canonical equivalence $\mc{C}^{op}[\mo{Poi}]\simeq \mc{C}[\mo{Poi}]^{op}$ taking the free map in $\mc{C}^{op}[\mo{Poi}]$ to the fiber of the free map $\unit\to\yo([1])$ in~$\mc{C}[\mo{Poi}]$.  In particular, the above proof run for $\mc{C}^{op}$ in place of $\mc{C}$ shows that there cannot be any nonzero object in $\mc{C}$ such that $\mo{fib}(\unit\to \yo[1])$ becomes $\otimes$-nilpotent after tensoring with this object.  Thus, the exact-nilpotence condition fails for $\mc{C}[\mo{Poi}]$, as claimed.
\end{proof}
\begin{remark}
It is tempting to attempt to perform a similar proof in the rigid setup.  The first obstruction to this idea is the fact that the free tt-category on a pointed object over a given tt-category $\mc{C}$ is inherently rather complicated.  Instead of a presheaf category on finite sets and injective maps, one has a presheaf category on a variant of the category of oriented 1-dimensional cobordisms where one allows half-open intervals.  Further complicating matters, this category is not going to be local, and its spectrum is rather complicated, similar to the spectrum of the rigid tt affine line, which will be studied in forthcoming work of Aoki, Barthel, Chedalavada, Stevenson and Schlank in \cite{HigherZar} (which will also study free constructions in the non-rigid case in more detail).  Nevertheless, if we understood all of the local categories in the free case over the category of spectra, this would lead either to a proof of the nerves of steel conjecture or a counterexample.
\end{remark}

\newpage
\section{Some Positive Results}
Although the exact-nilpotence condition does fail without rigidity, it may yet be the case that it always holds when working in a tt-category.  The goal of this section is to prove several cases where the condition does hold.  To begin, we will show the claim under some ``Noetherian'' type hypotheses.
\begin{proposition}\label{prop3.1}
If $R$ is a connective $\bb{E}_{\infty}$-ring such that $\pi_0(R)$ is a local ring, then $\mc{C}=\mo{Mod}_{R}^{perf}$ is a local tt-category.  If $\pi_*(R)$ is a Noetherian ring, then the exact-nilpotence condition holds for~$\mc{C}$.
\end{proposition}
\begin{proof}
Since $R$ is a connective $\bb{E}_{\infty}$-ring, \cite[~Proposition 7.1.3.15]{HA} shows that $\pi_0(R)=\tau_{\leq 0}R$ is canonically an $\bb{E}_{\infty}$-$R$-algebra.  This supplies us with a symmetric monoidal functor~$-\otimes_{R}\pi_0(R):\mo{Mod}_{R}^{perf}\to \mc{D}^b(\pi_0(R))$.  Since the target is local (e.g. \cite[~Example 4.4]{balmer2010spectra}), in order to see that $\mo{Mod}_{R}^{perf}$ is local, it suffices to show that if $M\otimes_{R}\pi_0(R)\simeq 0$, then~$M\simeq 0$.  From the fiber sequence $$\tau_{\geq 1}R\to R\to \pi_0(R),$$ we obtain a fiber sequence, for any $R$-module $M$, $$M\otimes_{R}\tau_{\geq 1}R\to M \to M\otimes_{R}\pi_0(R).$$  Now, if $M$ is a nonzero perfect $R$-module, there is some maximal $i\in\bb{Z}$ such that $\pi_j(M)=0$ for all $j<i$, and~$\pi_i(M)\neq 0$.  Since tensor products of connective modules over connective $\bb{E}_{\infty}$-rings remain connective, and $\tau_{\geq 1}R$ is $1$-connective, we have that $\pi_i(M\otimes_{R}\tau_{\geq 1}R)=0$, so that $M\otimes_{R}\tau_{\geq 1}R\to M$ cannot be an equivalence, and $M\otimes_{R}\pi_0(R)$ is nonzero as well, from which the claim follows.

Now, assume that $\pi_*(R)$ is Noetherian, let $k$ denote the residue field of $\pi_0(R)$, and write $\mf{m}$ for the maximal ideal in~$\pi_*(R)$.  We can take a presentation $\mf{m}=(a_1,...,a_n)$ for some a finite set of generators~$a_i$.  Define the perfect $R$-module $$z:=\bigotimes_{R,i=1}^{n}\mo{cofib}(\Sigma^{|a_i|}R\xrightarrow{a_i}R).$$  Although $a_i$ may not act as zero on $\pi_*(z)$, we at least know that $a_i^2$ acts as zero on the homotopy groups of~$z$.  We learn that $\pi_*(z)$ is a finitely generated graded module over the graded ring~$\pi_{*}(R)/(a_1^2,...,a_n^2)$.  In particular, $z$ only has finitely many nonzero homotopy groups, and $\pi_i(z)$ has a finite filtration whose associated graded pieces are all $k$-vector spaces, for all~$i$.

Now, consider a fiber sequence $$y\xrightarrow{g} \unit\xrightarrow{f}x$$ between perfect $R$-modules.  Since $k$ is a field, up to replacing this sequence by its dual, we may assume that $f\otimes_{R}k$ is split or equivalently that~$g\otimes_{R} k\simeq 0$.  We claim that $z\otimes_{R} g$ is $\otimes$-nilpotent.  Since $z$ has only finitely many nonzero homotopy groups, the Whitehead filtration on $z$ is a finite filtration, so it suffices to show that $\pi_i(z)\otimes_{R} g$ is $\otimes$-nilpotent for each~$i$.  Using that $\pi_i(z)$ has a finite filtration with associated graded pieces given by $k$-vector spaces, we deduce that $\pi_i(z)\otimes_{R} g$ is $\otimes$-nilpotent, since~$k\otimes_{R} g\simeq 0$.
\end{proof}
\begin{remark}
	In the above proof, when $\pi_*(R)$ is Noetherian, we constructed an object such that the exact-nilpotence condition for any fiber sequence always holds with said object playing the role of $z$ in the statement of the condition.  This can only happen when the closed point is the support locus of a single object.
\end{remark}
\begin{remark}
	The proof of Proposition~\ref{prop3.1} shows already that when $R$ is a connective $\bb{E}_{\infty}$-ring with $\pi_0(R)$ a local ring, the functor $-\otimes_{R}\pi_0(R):\mo{Mod}_R^{\omega}\to \mo{Mod}_{\pi_0(R)}^{\omega}$ is a local functor of local tt-categories.
\end{remark}
\begin{corollary}\label{cor3.2}
If $f:R\to S$ is a map of connective $\bb{E}_{\infty}$ rings inducing a local ring homomorphism of local rings on $\pi_0$, then $f^*:\mo{Mod}_R^{perf}\to \mo{Mod}_S^{perf}$ is a local functor of local tt-categories.
\end{corollary}
\begin{proof}
Given any nonzero perfect $R$-module $M\in \mo{Mod}_R^{perf}$, then $$M\otimes_{R}S\otimes_{S}\pi_0(S)\simeq M\otimes_{R}\pi_0(R)\otimes_{\pi_0(R)}\pi_0(S).$$Since $M\otimes_{R}\pi_0(R)$ is nonzero, the claim reduces to the case when $R=\pi_0(R)$ and $S=\pi_0(S)$ are classical local rings.  Any local map $R\to S$ of classical local rings induces a commutative diagram
\begin{center}
	\begin{tikzcd}
	R\rar\dar & S\dar\\
	k(R)\rar & k(S)
	\end{tikzcd}
\end{center}
where $k(R)$ (resp. $k(S)$) is the residue field of $R$ ($S$).  If $M$ is any nonzero perfect $R$-module, then $M\otimes_{R}k(R)$ remains nonzero, and thus $$M\otimes_{R}S\otimes_S k(S)\simeq M\otimes_{R}k(R)\otimes_{k(R)}k(S)\not\simeq 0.$$
\end{proof}

In \cite[~Theorem 1.4]{mathew2016residuefieldsclassrational}, Mathew computes the Balmer spectrum of the category of modules over rational $\bb{E}_{\infty}$-rings with Noetherian~$\pi_*$, and in doing so proves a nilpotence theorem.  Combining this with \cite[~Theorem 5.6]{Balmer2020}, we deduce:
\begin{proposition}\label{prop3.5}
The Nerves of Steel conjecture holds for the category $\mo{Mod}_R^{\omega}$ of perfect modules over a rational $\bb{E}_{\infty}$-ring $R$ with Noetherian $\pi_*$.  Equivalently, the exact-nilpotence condition holds for all localizations of such categories.
\end{proposition}
\begin{proof}
This follows from the nilpotence theorem in \cite[~Theorem 5.16]{mathew2016residuefieldsclassrational}, together with \cite[~Theorem 5.6]{Balmer2020}.  We should note that Mathew's notion of Noetherian \cite[~Definition 2.1]{mathew2016residuefieldsclassrational}, that $\pi_{even}(R)$ is Noetherian and $\pi_{odd}(R)$ is a finitely generated $\pi_{even}(R)$-module, takes a slightly different form to what we have written, stated in a way which acknowledges that $\pi_*(R)$ is only graded-commutative and not commutative on the nose.  However the two notions are equivalent.  Indeed, if $\pi_{even}(R)$ is Noetherian, with $\pi_{odd}(R)$ a finitely generated $\pi_{even}(R)$-module, then $\pi_*(R)$ is a finite type $\pi_{even}(R)$-module, so any increasing chain of submodules (e.g. ideals) has a maximal element, hence is Noetherian.  Conversely, if $\pi_*(R)$ is Noetherian, noting that for any ideal $I\subseteq \pi_{even}(R)$, ${(\pi_*(R)\cdot I)\cap \pi_{even}(R)=I}$, we find that $\pi_{even}(R)$ must be Noetherian too.  If $\pi_{odd}(R)$ is not finitely generated as a $\pi_{even}(R)$-module, we can find a sequence $x_1,\ldots,x_n,\ldots\in\pi_{odd}(R)$ of (homogeneous degree) elements with $x_{n+1}\notin \pi_{even}(R)\cdot (x_1,\ldots,x_n)$.  But now taking the ideals $I_n:=(x_1,\ldots,x_n)$, the fact that $\pi_*(R)$ is Noetherian implies that for some $n$, $x_{n+1}\in\pi_{*}(R)\cdot (x_1,\ldots, x_n)$, say $x_{n+1}=\sum_{i=1}^{n}r_ix_i$ with $r_i\in R$.  We can of course replace $r_i$ by its homogeneous component in degree $\mo{deg}(x_{n+1})-\mo{deg}(x_{i})$, but then $r_i\in \pi_{even}(R)$ for all $i$, a contradiction.  Therefore, $\pi_*(R)$ being Noetherian implies that $R$ is Noetherian in Mathew's sense as well.
\end{proof}
It is tempting to try to use Proposition~\ref{prop3.1} to prove that the exact-nilpotence condition holds for modules over any finitely presented local rational $\bb{E}_{\infty}$-ring, where we say that $R$ is finitely presented if $R$ is a compact object in the category of $\bb{E}_{\infty}$-$\bb{Q}$-algebras.  Unfortunately, we run into the same issue considered by Mathew.  From here on out we will say a rational $\bb{E}_{\infty}$-ring $R$ is Noetherian if $\pi_*(R)$ is Noetherian.
\begin{example}(\cite[~Proposition 8.8]{mathew2016residuefieldsclassrational})
The $\bb{E}_{\infty}$-ring $R\Gamma(\bb{A}^2_{\bb{Q}}\bs\{0\})$ is finitely presented, but not Noetherian.
\end{example}
One may hope the connective case is better.  The situation is not so serendipitous, and in fact, non-Noetherian compact $\bb{E}_{\infty}$-$\bb{Q}$-algebras are quite plentiful, as the following example shows.
\begin{example}
Start with the free $\bb{E}_{\infty}$-ring $\bb{Q}[x_2]\otimes \Lambda_{\bb{Q}}[y_1]$ on generators $y_1$ in degree $1$ and $x_2$ in degree~$2$.  Let $R$ denote the $\bb{E}_{\infty}$ quotient by $x_2y_1$, that is, $$\bb{Q}\otimes_{\Lambda_{\bb{Q}}[z_3]}(\bb{Q}[x_2]\otimes \Lambda_{\bb{Q}}[y_1]),$$ where $z_3$ maps to~$x_2y_1$.  Considering $\bb{Q}[x_2]\otimes \Lambda_{\bb{Q}}[y_1]$ as a $\Lambda_{\bb{Q}}[z_3]$-module, the generator $y_1$ generates a $\Sigma \bb{Q}$-summand, and when we take the tensor product, we will get generators in $\pi_{4n+1}$ which multiply with $x$, $y$ and each other to zero.  In particular, $\pi_*(R)$ is far from Noetherian.
\end{example}

In line with Balmer's vision that there should be a good notion of ``Noetherian'' in tt-geometry, one may expect that any finitely presented $\bb{E}_{\infty}$-$\bb{Q}$-algebra is ``Noetherian,'' whatever this should mean.  One expected property is that the Balmer spectrum of a Noetherian tt-category should be a Noetherian topological space.  We only prove this in the main case of interest, though remark that a careful examination of the proof will show how to adapt the following to prove such a claim in the connective case.  Before we begin, let's recall the notion of a collection of ring maps detecting nilpotence; the form we take is from \cite[~Definition 4.3]{mathew2016residuefieldsclassrational}.
\begin{definition}
Given an $\bb{E}_{\infty}$-ring $R$, and a collection $\{R\to S_i\}_{i\in I}$ of $\bb{E}_{\infty}$-ring maps, we say that this collection \textit{detects nilpotence} (over $R$) if given any associative algebra object $T$ in $\mo{Ho}(\mo{Mod}_R)$, then for any $x\in\pi_*(T)$, $x$ is nilpotent if and only if the image of $x$ under $\pi_*(T)\to\pi_*(S_i\otimes_{R}T)$ is nilpotent in $\pi_*(S_i\otimes_{R}T)$ for every~$i\in I$.
\end{definition}
For a detailed account about detecting nilpotence, we refer to \cite{mathew2016residuefieldsclassrational}.  We state the main properties that will be useful to us in the following.
\begin{proposition}\label{prop3.8}
Let $R$ be an $\bb{E}_{\infty}$-ring, and $\{R\to S_i\}_{i\in I}$ a collection of $\bb{E}_{\infty}$-ring maps which detects nilpotence.  The following hold:
\begin{enumerate}
	\item If $f:R\to A$ is an $\bb{E}_{\infty}$-ring map, then $\{A\to A\otimes_{R}S_i\}_{i\in I}$ detects nilpotence over~$A$.
	\item If $h:x\to y$ is a morphism in $\mo{Mod}_R^{perf}$, then $h$ is $\otimes$-nilpotent if and only if $h\otimes_{R}S_i$ is $\otimes$-nilpotent in $\mo{Mod}_{S_i}^{perf}$ for every~$i\in I$.
	\item The induced map on Balmer spectra $\coprod_{i\in I}\mo{Spc}(S_i)\to \mo{Spc}(R)$ is surjective.\footnote{Here and later on in the paper, a coproduct of Balmer spectra denotes the coproduct in the category of topological spaces.}
	\item If $\{S_i\to T_{ij}\}_{j\in J_i}$ detects nilpotence over each $S_i$, then $\{R\to T_{ij}\}_{i\in I, j\in J_i}$ detects nilpotence over~$R$.
	\item If $k$ is a field of characteristic zero, and $n\in\bb{Z}$, the collections $\{\Lambda_k[z_{2n+1}]\to k\}$ and $\{k[z_{2n}]\to k,k[z_{2n}]\to k[z_{2n}^{\pm 1}]\}$ detect nilpotence over $\Lambda_k[z_{2n+1}]$ and $k[z_{2n}]$, respectively.
	\item Given an $\bb{E}_{\infty}$-ring map $A\to B$, if the thick $\otimes$-ideal of $\mo{Mod}_A$ generated by $B$ is all of $\mo{Mod}_A$ (in which case we say that $A\to B$ admits descent), then $\{A\to B\}$ detects nilpotence over~$A$.
\end{enumerate}
\end{proposition}
\begin{proof}
For (a), if $T$ is any $A$-algebra, it is also an $R$-algebra, and there is an equivalence $T\otimes_{R}S_j\simeq T\otimes_{A}(A\otimes_{R}S_j)$, from which the result follows.  Parts (b), (d), (e), and (f), are Proposition 4.4, Proposition 4.6, Example 4.7, and Example 4.8 of \cite{mathew2016residuefieldsclassrational}, respectively.  Part (c) follows from \cite[~Theorem 1.3]{barthel2023surjectivitytensortriangulargeometry}.
\end{proof}
We will work explicitly with adjoining generators and relations to $\bb{E}_{\infty}$-$k$-algebras, for $k$ a field of characteristic zero.  By \cite[~Proposition 4.3.2.1]{Illusie}, \cite[~Construction 25.2.2.6]{SAG}, and \cite[~Proposition 25.1.2.2(3)]{SAG} for any $n\geq 0$, the free $\bb{E}_{\infty}$-$k$-algebra on a class $x$ in degree $n\geq 0$ is given by $\mo{LSym}_{k}(k[n])$.  For $n$ odd, when we denote this algebra by $\Lambda_k[x_n]$, this has underlying $k$-module structure $k\oplus k[n]$, with the class $x_n\in\pi_n(\Lambda_k[x_n])$ squaring to zero in the ring $\pi_*(\Lambda_k[x_n])$.  Given an $\bb{E}_{\infty}$-$k$-algebra $R$, a map $\Lambda_k[x_n]\to R$ is uniquely determined by a module map $k[n]\to R$, equivalently an element $a\in \pi_n(R)$, which becomes the image of $x_n$ under the induced algebra map.  If $n$ is even, then the free $\bb{E}_{\infty}$-$k$-algebra on a class in degree $n$ has the form $\mo{LSym}_k(k[n])$, which has underlying $k$-module structure $\bigoplus_{i\geq 0}k[ni]$, and we denote this algebra by $k[x_n]$.  The ring structure on $\pi_*(k[x_n])$ is that of a graded polynomial ring on a class $x_n$ in degree $n$, and a map $k[x_n]\to R$ to an $\bb{E}_{\infty}$-algebra $R$ is again uniquely determined by some $a\in \pi_n(R)$ serving as the image of $x_n$ when one takes $\pi_n$ of the induced algebra map.

With these preliminaries in hand, we can finally prove the following theorem.
\begin{theorem}\label{th3.9}
Let $R$ be a finitely presented connective $\bb{E}_{\infty}$-$\bb{Q}$-algebra with $\pi_0(R)=k$ a field.  Then there exists a finite collection $\{S_1,\ldots,S_n\}$ of $\bb{E}_{\infty}$-$R$-algebras under $R$ such that the collection $\{S_1,\ldots,S_n,k\}$ detects nilpotence for $R$, where each $S_i$ has a unit in $\pi_2(S_i)$, $\pi_1(S_i)=0$, and $\pi_0(S_i)$ is a $k$-smooth integral domain.
\end{theorem}
\begin{proof}
Let $R$ be a rational connective $\bb{E}_{\infty}$-algebra, with $\pi_0(R)$ a field~$k$.  Mathew proves that $R$ is a $k$-algebra \cite[~Proposition 2.14]{mathew2016residuefieldsclassrational}, and since $\pi_0$ commutes with colimits of connective $\bb{E}_{\infty}$-rings, $k$ is a finite extension of $\bb{Q}$, so is finitely presented.

Our goal is to build the collection $\{S_1,\ldots S_n\}$ by induction on the number of cells in $R$ if $R$ has finitely many cells.  The general claim will then follow since any finitely presented $R$ is a retract of some algebra with finitely many cells.  To explicitly build this, use that $\pi_i(R)$ is a finite dimensional $k$-vector space for any $i\geq 0$, and then build $T^{(0)}=k\to T^{(1)}\to\ldots\to R$ such that $T^{(i)}\to R$ is $i$-connected, and $T^{(i+1)}$ is obtained from $T^{(i)}$ by first adjoining new generators in degree $i+1$ and then quotienting out relations in this degree.  Then $R=\varinjlim T^{(i)}$, and compactness of $R$ implies that $T^{(n)}\to R$ splits for some~$n\gg 0$.

In the case $R=k$, we may take the empty collection, since $k$ detects nilpotence over~$k$.  Suppose by induction we have constructed such a collection $\{S_1,\ldots, S_n\}$ for $T$, and we attach to $T$ a new generator in degree $i$ or quotient out a relation in degree $i$ to get to~$R$.  We split into cases, the first three being the easiest:\\

\noindent\underline{\textbf{Case 1. $i$ is odd and we are adjoining a new generator in degree $i$.}}   Our new algebra is $R\simeq T\otimes_{k} \Lambda_k[t_i]$, with $t_i$ in degree~$i$.  Proposition~\ref{prop3.8}(f) implies that $\Lambda_{k}[t]\to k$ admits descent, and hence detects nilpotence.  Hence, $$T\otimes_{k}\Lambda_k[t_i]\otimes_{\Lambda_k[t_i]}k\simeq T$$ detects nilpotence over $R$, and we may take the same collection $\{S_1,\ldots, S_n\}$ used for $T$ for~$R$.\\

\noindent\underline{\textbf{Case 2.  $i$ is odd and we are adjoining a relation in degree~$i$.}}  In this case, there is a pushout diagram
\begin{center}
	\begin{tikzcd}
	\Lambda_k[t_i]\rar\dar & T\dar\\
	k\rar & R.
	\end{tikzcd}
\end{center}
By basechange compatibility (Proposition~\ref{prop3.8}(a)), the collection $S_j\otimes_{\Lambda_k[t_i]}k$ together with $k\otimes_{\Lambda_k[t_i]}k$ detect nilpotence over~$R$.  Since the $S_j$ are even, any map $\Lambda_k[t_i]\to S_j$ must factor over $k$, and $S_j\otimes_{\Lambda_k[t_i]}k$ is also even, with $$\pi_0(S_j\otimes_{\Lambda_k[t_i]}k)\simeq \pi_0(S_j)[x]$$ a polynomial ring over~$\pi_0(S_j)$.  Now, $k\otimes_{\Lambda_k[t_i]}k\simeq k[\sigma t_i]$ is a free algebra on a class $\sigma t_i$ in degree~$i+1$.  By Proposition~\ref{prop3.8}(e), the collection $\{k, k[(\sigma t_i)^{\pm 1}]\}$ detects nilpotence over~$k[\sigma t_i]$.  Now, we can use that there is a map $k[(\sigma t_i)^{\pm 1}]\to k[(x_2)^{\pm 1}]$ to the free algebra on an invertible class in degree 2 (taking $\sigma t_i$ to $x_2^{(i+1)/2}$) which admits descent, hence detects nilpotence.  Hence, we may take for $R$ the collection $$\{S_1\otimes_{\Lambda_k[t_i]}k,\ldots, S_n\otimes_{\Lambda_k[t_i]}k, k[x_2^{\pm 1}]\},$$ satisfying the desired properties.\\

\noindent\underline{\textbf{Case 3.  $i$ is even and we are adjoining a new generator in degree~$i$.}}  This is similar to Case 2.  In this case, $R=T\otimes_k k[t_i]$ with $t_i$ a polynomial generator in degree $i$, and we may use for $R$ the collection $$\{S_1\otimes_k k[t_i],\ldots, S_n\otimes_k k[t_i], k[x_2^{\pm 1}]\}.$$

\noindent\underline{\textbf{Case 4.  $i$ is even and we are adjoining a relation in degree~$i$.}}  In this case, there is an algebra map $k[t_i]\to T$, with~$R\simeq T\otimes_{k[t_i]}k$.  By basechange compatibility, the collection of $S_j\otimes_{k[t_i]}k$ and $k\otimes_{k[t_i]}k$ jointly detects nilpotence over~$R$.  The algebra $k\otimes_{k[t_i]}k\simeq \Lambda_{k}[\sigma t_i]$ is free on a generator in odd degree, so $k$ detects nilpotence over this algebra.

We reduce to showing that if $S$ is an even $\bb{E}_{\infty}$-$k$-algebra with a unit in $\pi_2(S)$ and with $\pi_0(S)$ a smooth integral domain over $k$, then for any $x\in \pi_0(S)$, there exists algebras $S_1^{\prime},\ldots,S_m^{\prime}$ with the same property, and maps $S/x\to S_j^{\prime}$ which jointly detect nilpotence.  By \cite[~Theorem 1.3/4]{mathew2016residuefieldsclassrational}, we find that a collection of maps $\{S/x\to S_j^{\prime}\}$ between even Noetherian $\bb{E}_{\infty}$-$k$-algebras with a unit in degree 2 detects nilpotence if and only if the induced map $$\coprod_{j=1}^{m}\mo{Spec}(\pi_0(S_j^{\prime}))\to \mo{Spec}(\pi_0(S/x))$$ is surjective.  Given an even 2-periodic $\bb{E}_{\infty}$-$k$-algebra $S$, there are a number of operations we can apply to $S$ make other algebras with these same properties.  Namely, we may adjoint a polynomial variable to $\pi_0(S)$, we may localize $\pi_0(S)$ at any multiplicatively closed set, and we may quotient out $\pi_0(S)$ by a regular sequence.  In particular, starting from our $S$, we can get maps to \'{e}tale covers, Zariski covers, and (affine covers of) blowups along smooth centers.

Consider the divisor cut out by $x\in\pi_0(S)$ on~$\mo{Spec}(S)$.  By Hironaka's theorem on embedded resolution of singularities \cite[~Corollary 3]{Hironaka}, there is a sequence of blowups along smooth centers $X_r\to\ldots \to X_0=\mo{Spec}(\pi_0(S))$ such that the (reduced subscheme structure on the) pullback of $(x)$ to $X_r$ is a normal crossing divisor.  By what we have said, $X_r$ has an affine open cover by schemes $\mo{Spec}(\pi_0(A_j))$ for even 2-periodic smooth $\bb{E}_{\infty}$-$S$-algebras~$A_j$.  Refining this by an \'{e}tale cover to assume $x$ pulls back to a strict normal crossing divisor, and then further if need be to assume that every irreducible component of the pullback of $x$ is given by a principal divisor, we may assume that we have a collection of maps $S\to A_j$ of even 2-periodic smooth $\bb{E}_{\infty}$-rings detecting nilpotence (since the induced map of Zariski spectra of their $\pi_0$ is jointly surjective) such that the image of $x$ in each $A_j$ is either a unit, or can be written as $x=y_1\ldots y_{r_j}$ with $\pi_0(A_j)/y_n$ a $k$-smooth integral domain for each~$n$.  Finally, we use that we have maps $S/x\to A_j/y_k$ which jointly detect nilpotence, again using the result of Mathew and the fact that the induced map on Zariski spectra is surjective.  Therefore, taking the collection $\{A_j/y_k\}_{j,k}$ as our $S_1^{\prime},\ldots, S_m^{\prime}$, the claim is shown.
\end{proof}
\begin{corollary}\label{cor3.10}
If $R$ is a connective finitely presented $\bb{E}_{\infty}$-$\bb{Q}$-algebra with $\pi_0(R)$ a field, then $\mo{Spc}(R)$ is a Noetherian topological space, and there is a collection of ``residue fields'' for~$R$.  These come in the form of $\bb{E}_{\infty}$-algebra maps $R\to L_j$ for $j\in J$ some index set, with each $L_j$ an even 2-periodic $\bb{E}_{\infty}$-$k$-algebra with $\pi_0(L_j)$ a field, such that the collection $\{R\to L_j\}\cup\{R\to k\}$ detects nilpotence over~$R$.
\end{corollary}
\begin{proof}
By Theorem~\ref{th3.9} and Proposition~\ref{prop3.8}(c), there is a surjection from the Noetherian topological space $\coprod_{i=1}^{n}\mo{Spc}(S_i)\coprod \mo{Spc}(k)\to \mo{Spc}(R)$, so that $\mo{Spc}(R)$ must be Noetherian as well, proving the first claim.  The second claim follows by Theorem~\ref{th3.9} and \cite[~Theorem 1.3/4]{mathew2016residuefieldsclassrational}.
\end{proof}
\begin{remark}
We only prove existence of residue fields in the above, not ``uniqueness.''  Below, we will prove ``uniqueness'' of the residue field at the closed point.  We do not say anything about whether or not there can be residue fields $R\to L_1$, $R\to L_2$, both having the same image (different from the closed point) under the induced map of Balmer spectra, but with~$L_1\otimes_{R}L_2\simeq 0$.  One would not expect this to happen, but we do not rule out the possibility in the present work.
\end{remark}
\begin{theorem}\label{th3.13}
If $R$ is a connective finitely presented $\bb{E}_{\infty}$-$\bb{Q}$-algebra with $\pi_0(R)=k$ a field, then the exact-nilpotence condition holds for~$\mo{Mod}_R^{perf}$.
\end{theorem}
\begin{proof}
Since $\mo{Spc}(R)$ is Noetherian, there is some object $Z\in\mo{Mod}_R^{perf}$ with $\mo{supp}(Z)$ equal to the unique closed point, we fix a choice of such an object.  Take a fiber sequence $$y\xrightarrow{g}\unit\xrightarrow{f}x$$ of perfect $R$-modules, and suppose without loss of generality (up to dualizing this sequence) that~$g\otimes_{R}k\simeq 0$.  We claim that $Z\otimes g$ is $\otimes$-nilpotent.  By Corollary~\ref{cor3.10}, there is a set of residue fields $\{L_j\}_{j\in J}\cup \{k\}$ for $R$, and so it suffices to check that $L_j\otimes_{R}(Z\otimes g)$ is $\otimes$-nilpotent for each of our constructed residue fields~$L_j$.  If the map $\mo{Spc}(L_j)\to \mo{Spc}(R)$ has image different from the closed point, then $L_j\otimes_{R}Z\simeq 0$, so we are reduced to the case when $L_j\otimes_{R}-$ has trivial kernel on perfect $R$-modules.  Over $L_j$, either $g\otimes_{R}L_j$ or $f\otimes_{R}L_j$ is zero, and we claim that $g\otimes_{R}L_j$ is zero, which would follow if we knew that $k\otimes_{R}L_j$ was nonzero.

Write $L[x_2^{\pm 1}]:=L_j$, with $L=\pi_0(L_j)$ a field and $x_2$ a chosen unit in $\pi_2(L_j)$.  Write $L[x_2]$ for the connective cover of $L[x_2^{\pm 1}]$ (noting $R\to L[x_2^{\pm 1}]$ factors over $R\to L[x_2]$ by connectivity of $R$).  Suppose towards a contradiction that~$L[x_2^{\pm 1}]\otimes_{R}k\simeq 0$.  First, we note that $\pi_*(R)\to \pi_*(L[x_2^{\pm 1}])$ must factor over $k$, or else there is some $y\in\pi_{2n}(R)$ mapping to a unit multiple of $x_2^n$ for some $n>0$, which would imply that $\mo{cofib}(y)\otimes_{R}L[x_2^{\pm 1}]\simeq 0$, contradicting the choice of map~$R\to L[x_2^{\pm 1}]$.  Suppose that we have inductively constructed an $\bb{E}_{\infty}$-$R$-algebra $R_i$ which is perfect as an $R$-module, with $R_0=R$, such that $R\to L[x_2]$ factors uniquely as~$R\to R_i\to L[x_2]$, and such that $\pi_*(R_i)\to \pi_*(L[x_2^{\pm 1}])$ factors over $k$.

If there is some $n>0$ with $\pi_{2n}(R_i)\neq 0$, take $n>0$ minimal with this property, choose a nonzero $z_{2n}\in \pi_{2n}(R_i)$, and define~$R_{i+1}:=R\otimes_{k[z_{2n}]}k$.  Upon applying $\mo{Hom}_{\mo{CAlg}}(-,L[x_2^{\pm 1}])$ to the pushout diagram \begin{center}
	\begin{tikzcd}
	k[z_{2n}]\rar\dar & R_i\dar\\
	k \rar & R_{i+1},
	\end{tikzcd}
\end{center}
we obtain a fiber sequence of spaces $$\mo{Hom}_{\mo{CAlg}}(R_{i+1},L[x_2^{\pm 1}])\to \mo{Hom}_{\mo{CAlg}}(R_{i},L[x_2^{\pm 1}])\to \mo{Hom}_{\mo{CAlg}}(k[z_{2n}],L[x_2^{\pm 1}]).$$  Since $\pi_1(\mo{Hom}_{\mo{CAlg}}(k[z_{2n}],L[x_2^{\pm 1}]))\simeq 0$, the extension of $R_i\to L[x_2^{\pm 1}]$ to $R_i\to R_{i+1}\to L[x_2^{\pm 1}]$ is unique. If $\pi_{2n}(R_i)=0$ for $n>0$, define~$R_{i+1}:=R_i$.

We claim that $R_{i+1}$ has the desired properties.  Indeed, suppose that the map $\pi_*(R_{i+1})\to \pi_*(L[x_2^{\pm 1}])$ has image larger than $k$, so that there is some element $y\in \pi_{2n}(R_{i+1})$ mapping to a unit multiple of $x_2^n$ for some~$n>0$.  Since $R_{i+1}$ is a perfect $R$-module by construction, it suffices to show that, should such a $y$ exist, then $$R_{i+1}/y\otimes_{R}L[x_2^{\pm 1}]\simeq 0,$$ or equivalently that $x_2$ is nilpotent in the algebra $$R_{i+1}/y\otimes_{R}L[x_2].$$   The splitting $R_{i+1}\otimes_{R}L[x_2]\to L[x_2]$ shows that the image of $y\otimes 1$ in $\pi_*(R_{i+1}\otimes_{R}L[x_2])$ is not nilpotent.  Again, by construction of $R_i$, $R_{i+1}\otimes_{R}L[x_2]$ is a perfect $L[x_2]$-module, so that the (graded) ring $$A:=\pi_{2n *}(R_i\otimes_{R}L[x_2])$$ is finitely generated as a module over the polynomial subring $L[y\otimes 1]$, and in particular, $1\otimes x_2^n$ is integral over this subring.  Whence, there is some monic polynomial $f(z)$ over $L[y\otimes 1]$, which must be homogeneous for $|y\otimes 1|=2n$, $|z|=2n$.  In particular, the polynomial $f(z)$ takes the form $$f(z)=z^j+a_{j-1}(y\otimes 1)z^{j-1}+\ldots + a_1(y\otimes 1)^{j-1}z+a_0(y\otimes 1)^{j},$$ for some $j>0$ and some~$a_i\in L$.  From this, it follows that $1\otimes x^{nj}\simeq 0$ in $R_i/y\otimes_{R}L[x_2]$, and hence $R_i/y\otimes_{R}L[x_2^{\pm 1}]\simeq 0$, contradicting the fact that $M\otimes_{R}L[x_2^{\pm 1}]$ is nonzero for every nonzero perfect $R$-module~$M$.

Our sequence $R_i\to R_{i+1}\to\ldots$ terminates only if $R_i$ has $\pi_{2n}(R_i)=0$ for all $n>0$ at some finite stage of the construction.  In any case, since $\pi_n(R)$ is finite dimensional for all $n>0$, the maps $\ldots \to R_i\to R_{i+1}\to\ldots$ become increasingly connective, and upon taking the colimit $A:=\varinjlim_{i}R_i$, one finds that $\pi_{2n}(A)=0$ for $n>0$, and $R\to L[x_2^{\pm 1}]$ factors uniquely as~$R\to A\to L[x_2^{\pm 1}]$.  Now, there is a map $A\otimes_{k} k[x_2^{\pm 1}]\to L[x_2^{\pm 1}]$, and the ring $A\otimes_{k}k[x_2^{\pm 1}]$ satisfies the hypothesis of \cite[~Proposition 4.9]{mathew2016residuefieldsclassrational} by construction.  Now, by Mathew's  Proposition~4.9, maps $$A\otimes_{k}k[x_2^{\pm 1}]\to L[x_2^{\pm 1}]$$ are in bijection with maps $$\pi_*(A\otimes_{k}k[x_2^{\pm 1}]\to L[x_2^{\pm 1}])\to \pi_*(L[x_2^{\pm 1}]).$$  In this way, we see that the map $A\otimes_{k}k[x_2^{\pm 1}]\to L[x_2^{\pm 1}]$ is homotopic to the composite $$A\otimes_{k}k[x_2^{\pm 1}]\to k\otimes_{k}k[x_2^{\pm 1}]\to L[x_2^{\pm 1}].$$  Finally, the following commutative diagram \begin{center}
	\begin{tikzcd}
		A\rar\dar & A\otimes_{k} k[x_2^{\pm 1}]\dar\\
		k\rar & k\otimes_{k}k[x_2^{\pm 1}],
	\end{tikzcd}
\end{center}
shows that $A\to L[x_2^{\pm 1}]$ factors as $A\to k\to L[x_2^{\pm 1}]$, so the same holds of~$R\to A\to k \to L[x_2^{\pm 1}]$.  But then $$k\otimes_{R}L[x_2^{\pm 1}]\simeq (k\otimes_{R}k)\otimes_{k}L[x_2^{\pm 1}],$$ which cannot be zero, as $k\otimes_{R}k$ is nonzero.  This yields the desired contradiction.
\end{proof}

Finally, we state the main theorems of the section, first showing how to get new examples where the exact-nilpotence theorem holds from known ones, then summarizing the known cases where the condition is known to hold.
\begin{theorem}\label{th3.14}
	The class of local tt-categories for which the exact-nilpotence condition holds is closed under filtered colimits along local transition maps.
\end{theorem}
\begin{proof}
Let $\mc{C}:=\varinjlim_{i\in I}\mc{C}_i$ be a filtered colimit along local transition maps of local tt-categories for which the exact-nilpotence condition holds.  Consider a fiber sequence $$y\xrightarrow{g}\unit \xrightarrow{f}x$$ in~$\mc{C}$.  Then, there exists some $i\in I$ and some $x^{\prime}\in\mc{C}_i$ such that $x$ is a summand of the image of~$x^{\prime}$.  Up to adding the complementary summand to $x$ and $y$ (which won't affect the nilpotence of $f,g$ tensored with a nonzero object), we may assume that~$x=x^{\prime}$.  There is some $j\in I$, $j>i$, and a map $\unit\xrightarrow{f^{\prime}}x^{\prime}$ in $\mc{C}_j$ having image $\unit\xrightarrow{f}x$ in~$\mc{C}$.  Taking $0\neq z\in\mc{C}_j$ such that $z\otimes f^{\prime}$ or $z\otimes \mo{fib}(f^{\prime})$ is $\otimes$-nilpotent, the image of $z$ has this same property in~$\mc{C}$.  As the transition maps between the $\mc{C}_i$ were local, the image of $z$ in $\mc{C}$ is nonzero, and the exact-nilpotence condition holds for~$\mc{C}$.
\end{proof}
\begin{theorem}\label{th3.15} The class of local tt-categories where the exact-nilpotence condition holds contains the following examples
	\begin{enumerate}
		\item Localizations of the categories $(\mo{Sp}_G^{gen})^{\omega}$ of compact objects in the category spectra of genuine equivariant $G$-spectra for a compact Lie group~$G$. \cite[~Corollary 5.10]{Balmer2020}
		\item The category $\mo{Mod}_R^{perf}$ of perfect complexes over an ordinary local ring~$R$. \cite[~Corollary 5.11]{Balmer2020}
		\item Localizations of any tt-category $\mc{C}$ with weakly Noetherian spectrum which is stratified \cite{barthel2023stratification}.
		\item Localizations of the stable category of finite-dimensional Lie superalgebra representations $\mo{stab}(\mc{F}_{(\mf{g},\mf{g}_{\overline{0}})})$ over a complex Lie superalgebra. \cite{hamil2024homologicalspectrumnilpotencetheorems}
		\item Localizations of the category of perfect modules over a rational $\bb{E}_{\infty}$-ring $R$ with $\pi_*(R)$ a Noetherian ring, Proposition~\ref{prop3.5}, follows from \cite{mathew2016residuefieldsclassrational}.
		\item The category of perfect modules over a connective $\bb{E}_{\infty}$-ring $R$ such that $\pi_*(R)$ is a Noetherian ring and $\pi_0(R)$ is local.  (Proposition~\ref{prop3.1}).
		\item The category of perfect modules over a rational connective $\bb{E}_{\infty}$-ring $R$ with $\pi_0(R)$ local.
	\end{enumerate}
\end{theorem}
\begin{proof}
For part (g), let $R$ be any connective rational $\bb{E}_{\infty}$-ring with $\pi_0(R)$ local.  Then, $R$ can be written as a filtered colimit of localizations of finitely presented $\bb{E}_{\infty}$-$\bb{Q}$-algebras, by first writing $R$ as an arbitrary filtered colimit of finitely presented $\bb{E}_{\infty}$-$\bb{Q}$-algebras, and localizing every term $S$ in the colimit at the pullback to $\pi_0(S)$ of the maximal ideal in~$\pi_0(R)$.  By Corollary~\ref{cor3.2}, a map $R\to S$ of connective local $\bb{E}_{\infty}$-rings induces a local map on their categories of perfect modules if $\pi_0(R)\to \pi_0(S)$ is a local ring homomorphism.  By \cite[~Corollary 4.8.5.13]{HA}, this gives a presentation of $\mo{Mod}_R^{perf}$ as a filtered colimit along local transition maps of categories $\mo{Mod}_S^{perf}$ with $S$ the localization of a finitely presented connective rational $\bb{E}_{\infty}$-ring, so we may assume that $R$ has these same properties Theorem~\ref{th3.14}.
	
Note that $\pi_0(R)$ is a Noetherian ring, hence we can write the maximal ideal as~$m=(x_1,\ldots,x_n)$.  The $\bb{E}_{\infty}$-ring $R/(x_1,\ldots,x_n)$ is a connective, finitely presented $\bb{E}_{\infty}$-$\pi_0(R)/m$-algebra, with $\pi_0(R)/m$ a field.  The module $R/(x_1,\ldots,x_n)$ is also perfect, so the exact-nilpotence condition holds for $\mo{Mod}_{R}^{perf}$ if and only if it holds for $\mo{Mod}_{R/(x_1,\ldots,x_n)}^{perf}$, where the claim follows by Theorem~\ref{th3.13}.
\end{proof}
\newpage 
\section{Strengthening Theorem 1.2}
Finally, we will explain how the statement of Theorem~\ref{th1.2} can be improved to include a bound on the order of nilpotence.  The idea for this comes from recent ultraproduct constructions in higher algebra \cite{barthel2020chromatichomotopytheoryasymptotically}\cite{levy2023categorifyingreducedrings}.  We recall the definitions here.
\begin{definition}\label{def4.1}
Fix a $\bb{N}$-indexed collection $\{\mc{C}_i\}$ of tt-categories, and some non-principal ultrafilter $\mc{U}$ on the natural numbers~$\bb{N}$.  Then the ultraproduct of this collection with respect to $\mc{C}_i$ is defined as: $$\prod_{\mc{U}}\mc{C}_i=\varinjlim_{I\in\mc{U}}\prod_{i\in I}\mc{C}_i.$$
\end{definition}
\begin{remark}\label{rem4.2}
Since products and filtered colimits of symmetric monoidal stable $\infty$-categories commute with passing to the homotopy category, and both operations make sense for triangulated categories, the above construction could be done equally well at the triangulated level, for the more finitely-inclined.
\end{remark}
The formalism of ultraproducts makes the following proof rather simple:
\begin{theorem}\label{th4.3}
	The following are equivalent,
	\begin{enumerate}
		\item The nerves of steel conjecture holds.
		\item For every local tt-category $\mc{T}$, the exact-nilpotence condition holds.
		\item There exists an integer $n$ such that for every local tt-category $\mc{T}$, and any fiber sequence as in Definition~\ref{def1.1}, there exists a nonzero object $z\in\mc{T}$ such that either $z\otimes g^{\otimes n}\simeq 0$ or~$z\otimes f^{\otimes n}\simeq 0$.
	\end{enumerate}
\end{theorem}
\begin{proof}
Theorem~\ref{th1.2} is the equivalence $(a)\iff (b)$, so we need only see that $(b)\implies (c)$.  Suppose otherwise, that we had a collection $\mc{C}_n$ of local tt-categories, such that for each $n$, there is a fiber sequence $$y_n\xrightarrow{g_n}\unit \xrightarrow{f_n}x_n$$ such that there exists $0\neq z_n\in \mc{C}_n$ with $z_n\otimes g_n^{\otimes n}\simeq 0$, but there is no $0\neq z_{n}^{\prime}\in \mc{C}_n$ with~$z_n^{\prime}\otimes g^{\otimes n-1}\simeq 0$.  Fix some non-principal ultrafilter $\mc{U}$ on the naturals, and let $\mc{T}:=\prod_{\mc{U}}\mc{C}_n$ be the ultraproduct of these categories.  Then, $\mc{T}$ remains a local tt-category, since given nonzero objects $(x_n),(y_n)\in\mc{T}$, the set of $n$ such that $x_n$ (resp. $y_n$) is nonzero is contained in $\mc{U}$, so too is their intersection, and since $\mc{C}_n$ were all local, $x_n\otimes y_n$ is nonzero when both terms are, hence $(x_n)\otimes (y_n)$ is a nonzero object of~$\mc{T}$.  Rigidity, and the stable/triangulated structure both happen pointwise.  Now, there is a fiber sequence 
$$(y_n)_n \xrightarrow{(g_n)_n}\unit \xrightarrow{(f_n)_n} (x_n)_n$$ in $\mc{T}$, where $(f_n)_n$ is not $\otimes$-nilpotent on any nonzero object, and neither is~$(g_n)_n$.  Indeed, if $(g_n)_n$ were $\otimes$-nilpotent on a nonzero object $z$, we could take $k$ such that $((g_n)_n)^{\otimes k}\otimes z\simeq 0$, and then choosing some $I\in \mc{U}$ with $z$ nonzero on $I$, and any $r\in I$ with $r>k$, we would be supplied with a nonzero object $z_r\in\mc{C}_r$ such that $z_r\otimes g_r^{\otimes k}\simeq 0$, contradicting the choice of~$\mc{C}_r$.
\end{proof}
\newpage
\printbibliography

\end{document}